\documentclass[12pt]{amsart}
\usepackage{amscd,amsmath,amsthm,amssymb,graphics}
\usepackage{pstcol,pst-plot,pst-3d}
\usepackage{lmodern,pst-node}
\usepackage{multicol}
\usepackage{epic,eepic}
\usepackage{amsfonts,amssymb,amscd,amsmath,enumerate,verbatim}
\psset{unit=0.7cm,linewidth=0.8pt,arrowsize=2.5pt 4}

\newpsstyle{fatline}{linewidth=1.5pt}
\newpsstyle{fyp}{fillstyle=solid,fillcolor=verylight}
\definecolor{verylight}{gray}{0.97}
\definecolor{light}{gray}{0.9}
\definecolor{medium}{gray}{0.85}
\definecolor{dark}{gray}{0.6}

\usepackage{float}
\usepackage{subcaption}
\usepackage{tikz}
\usetikzlibrary{patterns}

\usetikzlibrary{positioning}
\tikzset{main node/.style={circle,fill=blue!15,draw,minimum size=.5cm,inner sep=0pt},
}

\def\opn#1#2{\def#1{\operatorname{#2}}} 
\opn\projdim{proj\,dim} \opn\injdim{inj\,dim} \opn\rank{rank}
\opn\depth{depth} \opn\grade{grade} \opn\height{height}
\opn\reg{reg} \opn\lreg{lreg} \opn\ini{in} \opn\lpd{lpd}
\opn\link{link}\opn\fdepth{fdepth}
\opn\Ind{Ind}
\newtheorem{Theorem}{Theorem}[section]
 
 \newtheorem{corollary}[Theorem]{Corollary}
 \newtheorem{proposition}[Theorem]{Proposition}
 \newtheorem{remark}[Theorem]{Remark}
 
 \newtheorem{example}[Theorem]{Example}
 
 \newtheorem{definition}[Theorem]{Definition}

 \newcommand{\lr}[1]{\left\langle#1\right\rangle}
\newcommand{\K}[0]{\mathbb{K}}
\newcommand{\f}{\mathcal{F}}

\DeclareMathOperator{\dep}{depth}
\DeclareMathOperator{\s}{\mathcal{S}}
\DeclareMathOperator{\p}{\mathcal{P}}
\DeclareMathOperator{\g}{\mathcal{G}}
\opn\dis{dis}
\def\pnt{{\raise0.5mm\hbox{\large\bf.}}}

\opn\Lex{Lex}

\begin{document}
 \title{Quasi-forest simplicial complexes and almost Cohen-Macaulay}

 \author{ Chwas Ahmed, Amir Mafi* and Mohammed Rafiq Namiq }

\address{Ch. Ahmed, Department of Mathematics, College of Science, University of Sulaimani, Kurdistan Region, Iraq.}
\email{chwas.ahmed@univsul.edu.iq}

\address{A. Mafi, Department of Mathematics, University of Kurdistan, P.O. Box: 416, Sanandaj,
Iran.}
\email{a\_mafi@ipm.ir}

\address{M. R. Namiq, Department of Mathematics, College of Science, University of Sulaimani, Kurdistan Region, Iraq.}
\email{mohammed.namiq@univsul.edu.iq}


\keywords{Simplicial complex, Cohen-Macaulay simplicial complex, edge ideal\\
* Corresponding author}

\begin{abstract}
In this paper we study the quasi-forest simplicial complexes and we define the concept of simplicial $k$-cycle (denoted by $\s_k$) and simplicial $k$-point (denoted by $\p_k$).
We show that a simplicial complex $\Delta$ is quasi-forest if and only if it does not have any $\p_k$ and any $\s_k$ for $k\geq 3$.
Furthermore we characterize almost Cohen-Macaulay quasi-forest simplicial complexes. In the end we show that the cycle graph $G=C_n$ is almost Cohen-Macaulay if and only if $n=3,4,5,6,7,8,9,11$.
\end{abstract}

\maketitle

\section*{Introduction}
Throughout this paper, we assume that $R=K[x_1,...,x_n]$ is the polynomial ring in $n$ variables over a field $K$ and $G$ is a simple graph (without loops and multiple edges) with vertex set $V(G)=\{x_1,\dots,x_n\}$ and edge set $E(G)$. One associates to $G$ the edge ideal $I(G)$ of $R$ which is generated by all monomial $x_ix_j$ such that $\{x_i,x_j\}\in E(G)$. The independence complex and the clique complex of the graph $G$ are defined by $\Ind(G)=\{A\subseteq V(G)\mid A$ is an independence set in $G\}$ and $\Delta(G)=\{B\subseteq V(G)\mid B$ is a clique of $G\}$, respectively. Note that an independent set of $G$  is a subset $A$ of $V(G)$ such that none of its elements are adjacent and a clique of $G$ is a subset $B$ of $V(G)$ such that $\{x_i,x_j\}\in E(G)$ for all $x_i,x_j\in B$ with $i\neq j$. It easy to see that $\Delta(G)=\Ind(\overline{G})$, where $\overline{G}$ is the complement of $G$. Using the Stanley-Reisner correspondence, we can associate to $G$ the independent complex $\Ind(G)$, where $I_{\Ind(G)}=I(G)$. Hence the Stanley-Reisner ring of $\Ind(G)$ is $R/I(G)$.
The graph $G$ or the edge ideal $I(G)$ is called Cohen-Macaulay if $R/I(G)$ is Cohen-Macaulay. Cohen-Macaulay graphs were studied in \cite{V2, CRT}. A complete classification of Cohen-Macaulay graphs does not exist. Also, a graph $G$ or the edge ideal $I(G)$ is called almost Cohen-Macaulay if $R/I(G)$ is almost Cohen-Macaulay. We say that  $R/I(G)$ is almost Cohen-Macaulay when $\depth  R/I(G)\geq \dim R/I(G) -1$.  Almost Cohen-Macaulay rings have been studied in \cite{Ha, K1, K2, I, CTT, TM, TMA, MN2}.

Let $\Delta$ be a simplicial complex. A facet $F$ of $\Delta$ is called leaf, if there exists a facet $M$ of $\Delta$ with $F\neq M$ such that $N\cap F\subset M\cap F$ for all facet $N$ of $\Delta$ with $N\neq F$. If each subcomplex $\Gamma$ of $\Delta$ has a leaf, then $\Delta$ is called a forest. A simplicial complex $\Delta$ is called quasi-forest, if there exists an order $F_1,\ldots, F_r$ of the facets of $\Delta$ such that $F_i$ is a leaf of the simplicial complex $\langle F_1,\ldots, F_i\rangle$ for each $i=1,\ldots, r$. A free vertex is a vertex which belongs to precisely one facet. It is known that each leaf has a free vertex. But the converse is not true in general. It is clear that every forest is a quasi-forest. We say that the graph $G$ is quasi-forest when $\Ind(G)$ is a quasi-forest simplicial complex. The concept of quasi-forest has been studied in \cite{Z, Fa, HHZ, GPSY, GY}.

In this paper we define the concept of simplicial $k$-cycle (denoted by $\s_k$) and simplicial $k$-point (denoted by $\p_k$) and we give some examples.
We study the quasi-forest simplicial complex and we prove that a simplicial complex $\Delta$ is quasi-forest if and only if it does not have any $\p_k$ and any $\s_k$ for $k\geq 3$. Furthermore, we characterize almost Cohen-Macaulay quasi-forest simplicial complexes as a generalization of \cite[Proposition 2.3]{GPSY}. In the end we prove that the cycle graph $G=C_n$ is almost Cohen-Macaulay if and only if $n=3,4,5,6,7,8,9,11$.
For any unexplained notion or terminology, we refer the reader to
\cite{HH, V}. Several explicit examples were performed with help of the computer algebra systems Macaulay2 \cite{GS}.

\section{Preliminaries}
In this section, we recall some definitions and known results which are used in this paper.

A simplicial complex $\Delta$ on the vertex set $V=\{x_1,\ldots,x_n\}$ is a collection of subsets of $V$ such that (i) $\{x_i\}\in\Delta$ for every $1\leq i\leq n$, and (ii) if $F\in\Delta$ and $H\subseteq F$, then $H\in\Delta$. Each element $F$ of $\Delta$ is called a face of $\Delta$ and it is called an $i$-face when $|F|=i+1$. The dimension of a face $F$ is $|F|-1$ and the dimension of $\Delta$ is defined to be $\dim\triangle=d-1$, where $d=\max\{|F|\mid F\in\Delta\}$. A facet of $\Delta$ is maximal face (with respect to inclusion). The set $\f(\Delta):=\{F_1,\ldots, F_r\}$ is the set of all facets of $\Delta$. A simplicial complex $\Delta$ with the facets $F_1,\dots, F_r$ is denoted by $\Delta=\lr{F_1,\dots,F_r}$. The Stanley-Reisner ideal of $\Delta$ is $I_\Delta:=\lr{\prod_{x_i\in F}x_i\mid F\notin\Delta}$ and the quotient ring $K[\Delta]=R/I_{\Delta}$ is the Stanley-Reisner ring of $\Delta$ over a field $K$ where $R=K[x_1,\dots,x_n]$. A simplicial complex $\Delta$ is pure if every facet has the same cardinality. A simplicial complex $\overline{\Delta}=\lr{\overline{F}\mid F\in\f(\Delta)}$ is the complement of $\Delta$. The Alexander dual of $\Delta$, denoted by $\Delta^\vee$, is defined as $\Delta^\vee=\{V\setminus F\mid F\not\in\Delta\}$. The subcomplex $\Delta(i)=\{F\in\Delta\mid\dim F= i\}$ is called the pure $i$-skeleton of $\Delta$.
 Terai \cite{Te} proved that if $I$ is a square-free monomial ideal of $R$, then the Castelnuovo-Mumford regularity $\reg(I_{\Delta^\vee})=\projdim(R/I_\Delta)$.

Herzog, Hibi and Zheng \cite{HHZ} proved the following beautiful result:

\begin{proposition}
A simplicial complex $\Delta$ is quasi-forest if and only if $\projdim I(\overline{\Delta})=1$.
\end{proposition}

Now, since $I(\overline{\Delta})=I_{\Delta^\vee}$ (see \cite[Lemma 1.2]{HHZ}) and by using Terai's result we conclude that a simplicial complex $\Delta$ is quasi-forest if and only if $\reg I_\Delta=2$.

Recall that a simplicial complex is called flag, if all minimal nonfaces consist of two elements, equivalently, $I_{\Delta}$ is generated by quadratic monomials. By \cite[Lemma 3.2]{HHZ}, a quasi-forest simplicial complex is flag. A monomial ideal $I$ generated in degree $d$ has a linear resolution if and only if the Castelnuovo-Mumford regularity of $I$ is $\reg(I)=d$ (see \cite[Lemma 49]{Va}). Fr\"oberg \cite{F} proved that the edge ideal $I(G)$ has a linear resolution if and only if $\overline{G}$ is chordal. Recall that a graph $G$ is called chordal if each cycle of length $>3$ has a chord.

As before the independent simplicial complex of a graph $G$ is the clique complex of $\overline{G}$ and vice versa. One can rephrase \cite[Lemma 3.1]{HHZ} as follows:
\begin{Theorem}\label{T0}
	Let $G$ be a graph and $I_{\Delta}=I(G)$ be its edge ideal. Then
	\begin{enumerate}
		\item[(1)] $\Delta=\Ind(G)$;
		\item[(2)] $\overline{G}=\Delta(1)$;
		\item[(3)] $\Delta$ is quasi-forest if and only if $\overline{G}$ is chordal.
	\end{enumerate}
\end{Theorem}
Now, one can conclude that the independence complex $\Ind(G)$ is quasi-forest if and only if the edge ideal $I(G)$ has a linear resolution.

A Ferrers graph is a bipartite graph on two distinct vertex sets $X=\{x_1,...,x_n\}$ and $Y=\{y_1,\ldots,y_m\}$ such that if $x_iy_j$ is an edge of $G$, then so is $x_py_q$ for all $1\leq p\leq i$ and $1\leq q\leq j$.

Corso and Nagel \cite[Theorem 4.2]{CN} proved the following result:
\begin{Theorem}
Let $G$ be a bipartite graph without isolated vertices. Then its edge ideal has a 2-linear resolution if and only if $G$ is (up to a relabeling of the vertices) a Ferrers graph.	
\end{Theorem}
By the above theorem, one can conclude that if $G$ is a Ferrers graph, then $\reg(I(G))=2$. In particular, the independence complex $\Ind(G)$ is  quasi-forest.

Fr\"oberg  \cite{F} proved that the following result:
\begin{Theorem}\label{T00}
Let $\Delta$ be a $(d-1)$-dimensional quasi-forest. Then $f_{d-1}\leq n-d+1$ with equality if and only if $\K[\Delta]$ is Cohen-Macaulay, where $f(\Delta)=(f_0,f_1,\ldots,f_{d-1})$ is the $f$-vector of $\Delta$.
	
\end{Theorem}


\section{Relations on quasi-forest simplicial complex}

We start this section by the following definition:
\begin{definition}
An {\it induced $k$-cycle} in a graph $G$ is a 3-cycle or a chordless cycle of length $k\geq4$, we denote it by $C_k$. A {\it simplicial $k$-cycle}, denoted by $\s_k$, in $\Delta$ is an induced $k$-cycle $C_k$ in $\Delta(1)$ such that no more than two vertices of $C_k$ are in the same facet of $\Delta$.
\end{definition}

\begin{example}\label{E1}
(i) Let $\Delta$ be a simplicial complex which has the facets\\
$\{x_1,x_2,x_3\},\{x_2,x_4,x_6\},\{x_3,x_4,x_5\}$. The cycle $x_2x_3,x_3x_4,x_4x_2$ is $\s_3$ in $\Delta$.

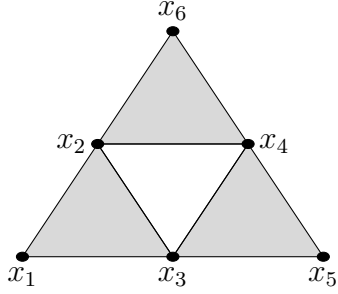
\begin{figure}[H]
		\begin{center}
			\begin{tikzpicture}[x = 2cm, y = 1.5cm]
			\path[fill = gray!30] (-1, 0) -- (-.5, 1) -- (0,0 ) -- (-1, 0);
			\path[fill = gray!30] (0,0) -- (.5,1) -- (1, 0)--(0,0);
			\path[fill = gray!30] (-.5,1) -- (.5, 1) -- (0,2)-- (-.5,1);

			\draw[fill] (-1,0) circle [radius = 0.04];
			\draw[fill] (0,0) circle [radius = 0.04];
			\draw[fill] (-.5,1) circle [radius = 0.04];
			\draw[fill] (.5,1) circle [radius = 0.04];
			\draw[fill] (1,0) circle [radius = 0.04];
			\draw[fill] (0,2) circle [radius = 0.04];
			
			\draw (-1, 0) -- (-.5, 1) -- (0,0 ) -- (-1, 0);
			\draw (0,0) -- (-0.5,1) -- (.5,1)--(0,0);
			\draw (0,0) -- (.5,1) -- (1, 0)--(0,0);
			\draw (-.5,1) -- (.5, 1) -- (0,2)-- (-.5,1);
			
			\node[below] at (-1,0) {$x_1$};
			\node[below] at (0,0) {$x_3$};
			\node[below] at (1,0) {$x_5$};
			\node[left] at (-.5,1) {$x_2$};
			\node[right] at (.5,1) {$x_4$};
			\node[above] at (0,2) {$x_6$};
		\end{tikzpicture}
		\end{center}
		\caption{Simplicial complex $\Delta$}
\end{figure}

(ii) Let $\Gamma$ be a simplicial complex which has the facets\\
$\{x_1,x_2,x_3\},\{x_2,x_4,x_6\},\{x_3,x_4,x_5\},\{x_2,x_3,x_4\}$. Then $\Gamma$ does not have any $\s_k$.

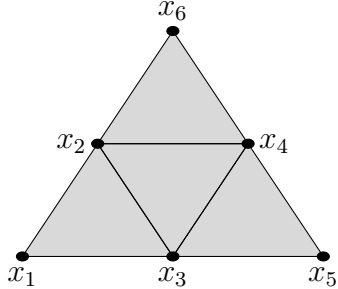
\begin{figure}[H]
		\begin{center}
				\begin{tikzpicture}[x = 2cm, y = 1.5cm]
				\path[fill = gray!30] (-1, 0) -- (-.5, 1) -- (0,0 ) -- (-1, 0);
				\path[fill = gray!30] (0,0) -- (-0.5,1) -- (.5,1)--(0,0);
				\path[fill = gray!30] (0,0) -- (.5,1) -- (1, 0)--(0,0);
				\path[fill = gray!30] (-.5,1) -- (.5, 1) -- (0,2)-- (-.5,1);

				\draw[fill] (-1,0) circle [radius = 0.04];
				\draw[fill] (0,0) circle [radius = 0.04];
				\draw[fill] (-.5,1) circle [radius = 0.04];
				\draw[fill] (.5,1) circle [radius = 0.04];
				\draw[fill] (1,0) circle [radius = 0.04];
				\draw[fill] (0,2) circle [radius = 0.04];
				
				\draw (-1, 0) -- (-.5, 1) -- (0,0 ) -- (-1, 0);
				\draw (0,0) -- (-0.5,1) -- (.5,1)--(0,0);
				\draw (0,0) -- (.5,1) -- (1, 0)--(0,0);
				\draw (-.5,1) -- (.5, 1) -- (0,2)-- (-.5,1);
				
				\node[below] at (-1,0) {$x_1$};
				\node[below] at (0,0) {$x_3$};
				\node[below] at (1,0) {$x_5$};
				\node[left] at (-.5,1) {$x_2$};
				\node[right] at (.5,1) {$x_4$};
				\node[above] at (0,2) {$x_6$};
			\end{tikzpicture}
		\end{center}
		\caption{Simplicial complex $\Gamma$}
\end{figure}
\end{example}

\begin{definition}
Let $\Delta$ be a simplicial complex. For $k\geq 3$, we say that $\Delta$  has {\it a simplicial $k$-point} $\p_k$ if there is a vertex $x_t\in V$ and a subset  $\{x_{j_1},\dots,x_{j_k}\}$ of  $V$ such that for each $1\leq i\leq k$ we have $\{x_t,x_{j_1},\dots,\hat{x}_{j_i},\dots,x_{j_k}\}\in\Delta$ and $\{x_t,x_{j_1},\dots,x_{j_k}\}\notin\Delta$. Also, we say that $\Delta$ has a simplicial $2$-point $\p_2$ if and only if it contains an $\s_3$.
\end{definition}

\begin{example}\label{E2}
	Let $\Theta$ be a simplicial complex as the following:
	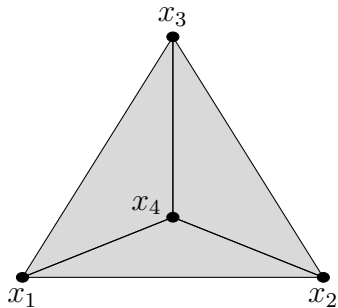
\begin{figure}[H]
		\begin{center}
				\begin{tikzpicture}[x = 1cm, y = .8cm]
			\path[fill = gray!30] (0, 0) -- (-2, -1) -- (0, 3) -- (0, 0);
			\path[fill = gray!30] (0, 0) -- (-2, -1) -- (2, -1) -- (0, 0);
			\path[fill = gray!30] (0, 0) -- (2, -1) -- (0, 3) -- (0, 0);
			
			\draw[fill] (0,0) circle [radius = 0.08];
			\draw[fill] (-2, -1) circle [radius = 0.08];
			\draw[fill] (2, -1) circle [radius = 0.08];
			\draw[fill] (0, 3) circle [radius = 0.08];
			
			\draw (-2, -1) -- (0, 0) -- (0,3);
			\draw (-2, -1) -- (0, 0) -- (2, -1);
			\draw (2, -1) -- (0, 0) -- (0, 3);
			\draw (-2, -1) -- (2, -1) -- (0, 3)-- (-2,-1);
			
			\node[below] at (-2, -1) {$x_1$};
			\node[above] at (0, 3) {$x_3$};
			\node[left] at (0, .2) {$x_4$};
			\node[below] at (2,-1) {$x_2$};
		\end{tikzpicture}
		\end{center}
			\caption{Simplicial complex $\Theta.$}
	\end{figure}
Then $\Theta$ has $\p_3$ since $x_4\in V$ and $\{x_1,x_2,x_4\}$, $\{x_1,x_3,x_4\}$, $\{x_2,x_3,x_4\}$ are in $\Theta$ and $\{x_1,x_2,x_3,x_4\}\notin\Theta$.
\end{example}

\begin{proposition}\label{P1}
If $\Ind(G)$ is a quasi-forest simplicial complex. Then $G$ does not contain an induced cycle $C_k$ for all $k\geq 5$.
\end{proposition}

\begin{proof}
Suppose $G$ contains an induced cycle $C_k:{x_{i_1}x_{i_2}},\dots,{x_{i_{k-1}}x_{i_k}},{x_{i_k}x_{i_1}},$ $k\geq 5$. Then we have two cases:
\begin{enumerate}
\item[(1)]: if $k=5$, then we have a cycle ${x_{i_1}x_{i_3}},{x_{i_1}x_{i_4}},{x_{i_2}x_{i_4}},{x_{i_2}x_{i_5}},{x_{i_3}x_{i_5}}$ in $\overline{G}$. Hence by Theorem \ref{T0}, $\Ind(G)$ is not quasi-forest and this is a contradiction.
\item[(2)]: if $k\geq 6$, then $x_{i_1}x_{i_2}$ and $x_{i_4}x_{i_5}$ are in $G$. Hence $\overline{G}$ contains a cycle $C_4$ and so $\overline{G}$ is not chordal. By Theorem \ref{T0}, $\Ind(G)$ is not quasi-forest and this is a contradiction.
\end{enumerate}
This completes the proof.
\end{proof}

From the above proposition one can deduce that every graph containing $C_k$ ($k\geq 5$) is not quasi-forest.

\begin{corollary}
If $\Ind(G)$ is quasi-forest and $G$ does not have an induced $3$-cycle and any isolated vertex, then $G$ is a bipartite graph. In particular, $G$ is a Ferrers graph.
\end{corollary}

\begin{proof}
By Proposition \ref{P1} and our hypothesis, $G$ does not have any odd cycle. Thus $G$ is a bipartite graph.
\end{proof}

Let $I$ be a monomial ideal, in the following we use $\g(I)$ the unique minimal set of monomial generators of $I$.
\begin{proposition}\label{C12}
A simplicial complex $\Delta$ is flag if and only if it does not have any $\p_k$ for $k\geq 2$.
\end{proposition}
\begin{proof}
	Suppose that $\Delta$ is not flag. Then there is a monomial $x_{j_1}\dots x_{j_k}\in\g(I_\Delta)$ for $k\geq 3$. It follows that $\{x_{j_2},x_{j_3},\dots,x_{j_{k}}\}$, $\{x_{j_1},x_{j_3},\dots,x_{j_k}\}$, $\dots$, $\{x_{j_1},x_{j_2}\dots,x_{j_{k-1}}\}$ are faces of $\Delta$ and $\{x_{j_1},\dots,x_{j_k}\}\notin\Delta$. Hence $\Delta$ has a $\p_{k-1}$ for $k\geq3$.
	
	Conversely, assume $\Delta$ has a $\p_{k}$. If $k=2$, then there are faces $\{x_{j_1},x_{j_2}\},\{x_{j_2},x_{j_3}\}$ and $\{x_{j_1},x_{j_3}\}$ of $\Delta$ such that $\{x_{j_1},x_{j_2},x_{j_3}\}\notin\Delta$. Hence $x_{j_1}x_{j_2}x_{j_3}\in\g(I_\Delta)$. If $k\geq 3$, then there is a vertex $x_t\in V$ and  $\{x_{j_1},\dots,x_{j_{k}}\}\subset V$ such that $\{x_t,x_{j_1},\dots,\hat{x}_{j_i},\dots,x_{j_{k}}\}$ are faces of $\Delta$ for $i=1,\dots,k$. If $\{x_{j_1},\dots,x_{j_{k}}\}\notin\Delta$, then $x_{j_1}\dots x_{j_{k}}\in\g(I_\Delta)$. Now if $\{x_{j_1},\dots,x_{j_{k}}\}\in\Delta$, then $x_tx_{j_1}\dots x_{j_{k}}\in\g(I_\Delta)$ since $\{x_t,x_{j_1},\dots,x_{j_{k}}\}\notin\Delta$. This completes the proof.
\end{proof}

\begin{remark}
Consider Examples \ref{E1} and \ref{E2}, the simplicial complexes $\Delta$ and $\Theta$ are not flag since they contain an $\s_3$ and a $\p_3$, respectively. However, the simplicial complex $\Gamma$ is flag since $\Gamma$ does not contain any $\p_k$.
\end{remark}

\begin{Theorem}\label{T11}
A simplicial complex $\Delta$ is quasi-forest if and only if it does not have any $\p_k$ and any $\s_k$.
\end{Theorem}

\begin{proof}
$(\Longrightarrow)$. Suppose $\Delta$ has a $\p_k$ such that $k\geq 2$. By Proposition \ref{C12} we obtain that $\Delta$ is not flag. Hence $\Delta$ is not quasi-forest and this is a contradiction. Suppose that $\Delta$ contains an $\s_k$ for $k\geq 4$. Then $C_k$ is in $\Delta(1)$. Hence $\overline{G}$ is not chordal. Hence by Theorem \ref{T0} it follows that $\Delta$ is not quasi-forest and this is a contradiction. Therefore $\Delta$ does not have any $\p_k$ and any $\s_k$.\\
$(\Longleftarrow)$. Suppose, by contrary, that $\Delta$ is not quasi-forest. Thus $\Delta$ is not flag or $\Delta(1)$ is not chordal. If $\Delta$ is not flag, then Proposition \ref{C12} implies that $\Delta$ has a $\p_k$ for $k\geq2$, a contradiction. If $\Delta(1)$ is not chordal, then $\Delta(1)$ contains $C_k, k\geq 4$. By assumption we may assume that at least three vertices of $C_k$ is a face in $\Delta$, say $F$. Then the 1-faces of $F$ are in $\Delta(1)$. It follows that $C_k$ has a chord and this is a contradiction. Therefore $\Delta$ is quasi-forest.
\end{proof}

\begin{remark}
The simplicial complexes $\Delta$ and $\Theta$ in Examples \ref{E1}, \ref{E2} are not quasi-forest since they contain an $\s_3$ and a $\p_3$, respectively. The simplicial complex $\Gamma$ in Example \ref{E1} is quasi-forest since it does not have any $\p_k$ and any $\s_k$.
\end{remark}
The following result immediately follows by Theorem \ref{T11}.
\begin{corollary}\label{T12}
	A simplicial complex $\Delta$ is forest if and only if any subcomplex $\Gamma$ of $\Delta$ does not have a $\p_k$ and an $\s_k$.
\end{corollary}
In Example \ref{E1}, the simplicial complex $\Gamma$ is not forest since if we remove the facet $\{x_2,x_3,x_4\}$, then $\Gamma$ is equal to $\Delta$ and $\Delta$ contains an $\s_3$.

\begin{definition}[{\cite[Definition 2.15]{Z},\cite[Definition 2.1]{DHS}}]
Let $G$ be a graph. Two edges $xy$ and $zu$ form a {\it gap} in $G$ if $G$ does not have an edge with one endpoint in $xy$ and the other in $zu$.
A graph without gaps is called {\it gap-free}. Equivalently, $G$ is gap-free if and only if $C_4$ is not an induced subgraph of $\overline{G}$.
\end{definition}

The following result immediately follows by Theorem \ref{T11}.

\begin{corollary}\label{C14}
If $\Delta=\Ind(G)$ is quasi-forest, then $G$ is a gap-free graph.
\end{corollary}

We recall the definition of a {\it perfect graph} and a {\it Berge graph} introduced in \cite{CRST}: A graph $G$ is perfect if for every induced subgraph $H$, the chromatic
number of $H$ equals the size of the largest complete subgraph of $H$, and $G$ is Berge if and only if neither it nor its complementary graph has an odd induced cycle of length at least five. Chudnovsky et al. in \cite{CRST} proved that a graph is perfect if and only if it is Berge.

\begin{corollary}\label{C15}
If $\Delta=\Ind(G)$ is a quasi-forest, then $G$ is a perfect graph.
\end{corollary}
\begin{proof}
From Proposition \ref{P1} and Theorem \ref{T11}, it follows that $G$ is a Berge graph and so $G$ is a perfect graph.
\end{proof}

Note that the converse of Corollaries \ref{C14}  and \ref{C15} are not true in general, for example, the graph $G$ of independence complex $\Ind(G)=\s_4$ is gap-free and perfect, however, $\Ind(G)$ is not quasi-forest.

Following \cite{W} a simplicial complex $\Delta$ is recursively defined to be {\it vertex decomposable} if it is either a simplex or else has some vertex $v$ so that
\begin{enumerate}
\item[(1)] both $\Delta\setminus v$ and $\link_{\Delta}v$ are vertex decomposable, and
\item[(2)] no face of $\link_{\Delta}v$ is a facet of $\Delta\setminus v$, where $\link_{\Delta}v=\{F\in\Delta\mid F\cup\{v\}\in\Delta, v\notin F\}.$
\end{enumerate}

Woodroofe in \cite[Theorem 1]{W} proved the following theorem:

\begin{Theorem}\label{T1}
If $G$ is a graph with no chordless cycle of length other than $3$ or $5$, then $G$ is vertex decomposable (hence shellable and sequentially Cohen-Macaulay).
\end{Theorem}

The following result immediately follows form Theorem \ref{T1} and Proposition \ref{P1}.

\begin{corollary}
If $\Delta=\Ind(G)$ is quasi-forest and $G$ does not contain an induced $4$-cycle, then $\Delta$ is a vertex decomposable (so shellable and sequentially Cohen-Macaulay).
\end{corollary}

It can be noted form Corollary \ref{T12} that if $\Delta$ is a quasi-forest and contains at most three facets, then $\Delta$ is a forest.

\begin{corollary}
The simplicial complex $\Ind(C_k)$ is quasi-forest if and only if $k=3, 4$. In particular, $\Ind(C_k)$ is a forest if and only if $\Ind(C_k)$ is a quasi-forest.
\end{corollary}
\begin{proof}
If $k=3,4$, then it is immediately follows that $\Ind(C_k)$ is quasi-forest. Conversely, we consider the following cases:
\begin{enumerate}
\item[(1)] if $k=3$, then $\reg(I_{\Ind(C_3)})=2$ and $\Ind(C_3)=\lr{x_1,x_2,x_3}$. Thus it is clear that $\Ind(C_3)$ is a forest;
\item[(2)] if $k=4$, then $\reg(I_{\Delta(C_4)})=2$ and $\Ind(C_4)=\lr{x_1x_3,x_2x_4}$. It therefore follows that $\Ind(C_4)$ is a forest;
\item[(3)] if $k\geq 5$, then by Proposition \ref{P1} we get the result.
\end{enumerate}
\end{proof}

\section{The $h$-vector of Cohen-Macaulay quasi-forest}

Let $I$ be a homogeneous ideal of $R$ with $\dim R/I=d$. The Hilbert series of $R/I$ is of the form $H_{R/I}(t)=(h_0+h_1t+h_2t^2+\ldots+h_st^s)/{(1-t)^d}$, where each $h_i\in\mathbb{Z}$. The polynomial $h_{R/I}(t)=h_0+h_1t+h_2t^2+\ldots+h_st^s$ with $h_s\neq 0$ is called the {\it h-polynomial} of $R/I$ (see \cite[Theorem 6.1.3]{HH}). The {\it a-invariant} is the degree of the Hilbert series $H_{R/I}(t)$, that is, the number $s-d$.

From \cite[Corollary B.4.1]{Vs} we have $a(K[\Delta])\leq\reg(K[\Delta])-\dep(K[\Delta])$ with equality if $K[\Delta]$ is Cohen-Macaulay. Then we have the following inequality $\deg h_{K[\Delta]}(t)-\reg(K[\Delta])\leq\dim(K[\Delta])-\dep(K[\Delta]).$ The equality holds if $K[\Delta]$ is Cohen-Macaulay or $K[\Delta]$ has a pure resolution (see \cite[P. 153]{BH}). In particular, if $\Delta$ is a quasi-forest, then the equality holds.

\begin{proposition}(Compare with \cite[Proposition 2.3]{GPSY})\label{P3}
Let $\Delta$ be a $(d-1)$-dimensional simplicial quasi-forest with $h(\Delta)=(h_0,h_1,\dots,h_d)$. Then $\Delta$ is Cohen-Macaulay if and only if $h_{2},\dots,h_d$ are zero.
\end{proposition}

\begin{proof}
Since $\Delta$ is a quasi-forest, we have
$\deg h_{K[\Delta]}(t)=\dim(K[\Delta])-\dep(K[\Delta])+1.$
If $\Delta$ is Cohen-Macaulay, then $\deg h_{\K[\Delta]}(t)=1$. Therefore $h_2,\dots, h_d$ are zero.
Conversely, let $h_2,\dots, h_d$ be zero. Then
$1=\deg h_{K[\Delta]}(t)=\dim(K[\Delta])-\dep(K[\Delta])+1.$
Thus $\dim(K[\Delta])=\dep(K[\Delta])$ and so $\Delta$ is Cohen-Macaulay.
\end{proof}

\begin{Theorem}
Let $\Delta$ be a $(d-1)$-dimensional simplicial quasi-forest with $h(\Delta)=(h_0,h_1,\dots,h_d)$. Then $\Delta$ is almost Cohen-Macaulay if and only if $h_2$ is non-positive, and $h_{3},\dots,h_d$ are zero.
	
\end{Theorem}

\begin{proof}
Since $\Delta$ is a quasi-forest, we have
	$\deg h_{K[\Delta]}(t)=\dim(K[\Delta])-\dep(K[\Delta])+1.$
	If $\Delta$ is almost Cohen-Macaulay, then $\deg h_{K[\Delta]}(t)\leq 2$. Therefore
	$h_3,\dots, h_d$ are zero. By \cite[Theorem 6.7.6]{V} $h_0=1$ and $h_1=f_0-d=n-d$ which is not negative. Now, we consider the following cases:
\begin{enumerate}
\item[(i)]: let $h_1=0$. Then $f_0=d$ and so $h_2$ is zero;
\item[(ii)]: let $h_1$ be positive. By Theorem \ref{T00},  we have $f_{d-1}\leq n-d+1$. Since $f_{d-1}=h_0+h_1+\dots+h_d$, it follows that $h_0+h_1+\dots+h_d\leq n-d+1$. Thus $h_2+\dots+h_d\leq 0$. It therefore follows that $h_2\leq 0$.
\end{enumerate}
	
	Conversely,  suppose that $h_2\leq 0$ and $h_3=\dots=h_d=0$.
	Since  $\dim(K[\Delta])-\dep(K[\Delta])+1=\deg h_{K[\Delta]}(t)\leq 2,$
	it follows that $\dim(K[\Delta])-\dep(K[\Delta])\leq 1$.
	Hence $\Delta$ is almost Cohen-Macaulay, as required.
\end{proof}
As before if $K[\Delta]$ is Cohen-Macaulay, then $a(K[\Delta])=\reg(K[\Delta])-\dep(K[\Delta])$. It is natural to ask whether if $K[\Delta]$ is almost Cohen-Macaulay, then $a(K[\Delta])=\reg(K[\Delta])-\dep(K[\Delta])-1$ is it true in general.
In the following we give a counter example for this question.

\begin{example}
Let $n=5$ and $I=(x_4x_5,x_1x_3x_5,x_1x_2x_5,x_1x_2x_3x_4)$ be an ideal of $R$. Then by using Macaulay 2 we have $\dim(R/I)=3$ and $\dep(R/I)=2$ and so $I$ is almost Cohen-Macaulay. On the other hand, $\reg(R/I)=3$ and $H_{R/I}(t)=\frac{1+2t+2t^2}{(1-t)^3}$. Thus $a(R/I)=-1$ and $\reg(R/I)-\dep(R/I)-1\neq a(R/I)$.
\end{example}

\begin{proposition}(Compare with \cite[Corollary 2.8]{CN})	
Let $G$ be a Ferrers graph on two distinct vertex sets $V_1=\{x_1,\ldots,x_n\}$ and $V_2=\{y_1,\ldots,y_n\}$ and $I=I(G)$ be the edge ideal in the ring $S=K[x_1,\ldots,x_n,y_1,\ldots,y_n]$. Then $I$ is Cohen-Macaulay if and only if the Hilbert series $H_{S/I}(t)=\frac{1+nt}{(1-t)^n}$.
\end{proposition}

\begin{proof}
Let $I$ be Cohen-Macaulay. Then by Proposition \ref{P3} $h_2,\dots, h_d$ are zero and $d=\dim R/I=n$. Since $h_1=f_0-d=2n-n=n$, it follows that
$H_{S/I}(t)=\frac{1+nt}{(1-t)^n}$. Conversely, suppose $H_{S/I}(t)=\frac{1+nt}{(1-t)^n}$. Since $I$ has a linear resolution, we have
$\deg h_{S/I}(t)-\reg(S/I)=\dim(S/I)-\dep(S/I)$ and $\reg(S/I)=1$. By hypothesis $\deg h_{S/I}(t)=1$ and it therefore follows $\dim(S/I)=\dep(S/I)$
Thus $I$ is Cohen-Macaulay, as required.
\end{proof}

The following result proved by Cimpoeas in \cite[Proposition 1.3]{C}:
\begin{proposition}\label{P22}
Let $G=C_n$ be a cycle graph and $I=I(G)$. Then $\depth R/I=\lceil\frac{n-1}{3}\rceil$.
\end{proposition}

\begin{Theorem}
Let $G=C_n$ be a cycle graph and $I=I(G)$. Then $I$ is almost Cohen-Macaulay if and only if $n=3,4,5,6,7,8,9,11$.
\end{Theorem}

\begin{proof}
$(\Longleftarrow)$. By using Macaulay 2 one can easily obtain the result.\\
$(\Longrightarrow)$. Let $J=I(P_n)$, where $P_n$ is a path of length $n-1$ with $n-1$ edges $x_ix_{i+1}$ such that $1\leq i\leq n-1$. Now by induction on $n$, we prove that $\dim(R/J)=\lfloor\frac{n}{2}\rfloor$. If $n=3$, then by using Macaulay 2 there is nothing to prove. Suppose $n\geq 4$ and the result has been proved for smaller values of $n$. Consider the exact sequence
\[0\longrightarrow R/(J:x_n)\overset{x_n}\longrightarrow R/J\longrightarrow R/(J,x_n)\longrightarrow 0.\ \ (\ast)
\]
Since $(J:x_n)=(I(P_{n-2}),x_{n-1})$, it follows that $\dim R/(J:x_n)=\dim R^{'}[x_n]/I(P_{n-2})\\=\dim R^{'}/I(P_{n-2})+1$, where $R^{'}=K[x_1,\ldots,x_{n-2}]$. Now, by induction hypothesis, we obtain $\dim R/(J:x_n)=\lfloor\frac{n-2}{2}\rfloor+1=\lfloor\frac{n}{2}\rfloor$.
Similarly, $(J,x_n)=(I(P_{n-1}),x_n)$ and again by using induction hypothesis $\dim R/(J,x_n)=\dim K[x_1,\ldots,x_{n-1}]/(I(P_{n-1}))=\lfloor\frac{n-1}{2}\rfloor$.
Since $\dim R/J=\max\{\dim R/(J:x_n),\dim R/(J,x_n)\}$, it follows that $\dim R/J=\lfloor\frac{n}{2}\rfloor$. Since $(I:x_n)=(I(P_{n-3}),x_{n-1},x_1)$, similarly by using $I$ in stead of $J$ in the exact sequence $(\ast)$, we get $\dim R/I=\lfloor\frac{n}{2}\rfloor$. Also, by Proposition \ref{P22} we have
$\depth R/I=\lceil\frac{n-1}{3}\rceil$. Hence by comparing $\depth R/I$ and $\dim R/I$ we conclude that $n=3,4,5,6,7,8,9,11$.
\end{proof}


\end{document}